\documentclass[11pt]{article}
\usepackage{amsmath,amssymb,amsthm}
\usepackage{graphicx}

\newtheorem{theorem}{Theorem}[section]

\newtheorem{lemma}[theorem]{Lemma}

\newtheorem{definition}[theorem]{Definition}
\providecommand{\func}[1]{\operatorname{#1}}
\providecommand{\limfunc}[1]{\operatorname{#1}}
\begin{document}

\title{A computable wandering and tracelike vector for modular orbits in the
Bergman space}
\author{Lu\'{\i}s Daniel Abreu}
\date{}
\maketitle

\begin{abstract}
We \emph{construct} a function $\Phi $ such that the orbit under the
representation of $\Gamma =\func{PSL}(2,\mathbb{Z})$ is an orthonormal basis
for the Bergman space with weight $\alpha =12$. Moreover, we show that $\Phi 
$ is \emph{effectively computable} as a holomorphic function on the upper
half-plane (in the precise sense of computable analysis), by providing an
effective procedure. This constructs a wandering and tracelike vector for $%
\func{PSL}(2,\mathbb{Z})$, whose abstract existence was proved by Sir
Vaughan Jones in his last paper, where the corresponding construction was
left as a problem. The function is built using an orthonormalization and modularization method, and it displays modular reminiscencies, despite not being modular itself. It provides a computable implementing vector
for the abstract anti-isomorphism between the von Neumann algebra $%
M_{12}(\Gamma )$ and its commutant, which is generated, in R\u{a}dulescu's
sense, by cusp-form Toeplitz operators, while Voiculescu's results provide a
random matrix model for $M_{12}(\Gamma )$.
\end{abstract}

\section{Introduction}

\subsection{The Bergman space of the upper half-plane}

Let $\mathbb{H}=\{z:\func{Im}z>0\}$. For $\alpha >0$, the Bergman spaces in $%
\mathbb{H}$ are obtained by a conformal map from the Bergman spaces in the
unit disc \cite{HKZ}. Define the Bergman space $\mathcal{A}_{\alpha }=%
\mathcal{A}_{\alpha }\left( \mathbb{H}\right) $ by 
\begin{equation*}
\mathcal{A}_{\alpha }=\left\{ F\text{ holomorphic on }\mathbb{H}:\left\Vert
F\right\Vert _{\mathcal{A}_{\alpha }}^{2}=\int_{\mathbb{H}}|F(z)|^{2}(\func{%
Im}z)^{\alpha -1}\,dA_{\alpha }(z)<\infty \right\} \text{,}
\end{equation*}%
where $dA_{\alpha }(z)=\frac{\alpha 2^{\alpha -1}}{\pi }dx\,dy$. This is a
Hilbert space with reproducing kernel 
\begin{equation*}
K_{\alpha }(z,w)=\left( \frac{i}{z-\overline{w}}\right) ^{\alpha +1}\text{.}
\end{equation*}%
The following rational functions provide an orthonormal basis for $\mathcal{A%
}_{\alpha }$: 
\begin{equation}
e_{n}^{\alpha }(z)=\left( \frac{\Gamma (n+\alpha +1)}{\Gamma (\alpha
+1)\Gamma (n+1)}\right) ^{1/2}\left( \frac{z-2i}{z+2i}\right) ^{n}\left( 
\frac{2i}{z+2i}\right) ^{\alpha +1},\qquad n\geq 0\text{.}  \label{basis}
\end{equation}

\subsection{The modular action of $\func{PSL}(2,\mathbb{Z})$}

Let 
\begin{equation*}
\Gamma =\func{PSL}(2,\mathbb{Z})=\limfunc{SL}(2,\mathbb{Z})/\{I,-I\}\text{.}
\end{equation*}%
For every $\gamma \in \Gamma =\func{PSL}(2,\mathbb{Z})$, let 
\begin{equation*}
\sigma (\gamma )=%
\begin{pmatrix}
a_{\gamma } & b_{\gamma } \\ 
c_{\gamma } & d_{\gamma }%
\end{pmatrix}%
\end{equation*}%
denote the unique representative of $\gamma $\ satisfying $c_{\gamma }>0,$
or $c_{\gamma }=0,\ d_{\gamma }=1$.\ Since $ad-bc=1$ the integers $c$ and $d$
are coprime. Thus, when $c>1$ one can define the Dedekind sum%
\begin{equation*}
s(d,c)=\sum_{r=1}^{c-1}\left( \frac{r}{c}-\frac{1}{2}\right) \left( \frac{%
\left[ dr\right] _{c}{}}{c}-\frac{1}{2}\right) \text{, \ \ \ }c>0\text{,}
\end{equation*}%
where $\left[ dr\right] _{c}\in \{1,...,c-1\}$ denotes the residue of $dr$
modulo $c$. When $c=1$ this is an empty sum and $s(d,c)=0$. Consider then
the character%
\begin{equation*}
\chi 
\begin{pmatrix}
a & b \\ 
c & d%
\end{pmatrix}%
=\left\{ 
\begin{array}{cc}
e^{2\pi i\left[ \frac{a+d}{12c}-s(d,c)-\frac{1}{4}\right] ,} & c>0 \\ 
e^{\frac{\pi ib}{6}} & c=0,\text{ \ }d=1%
\end{array}%
\right.
\end{equation*}%
and, for every integer $\alpha >0$, let%
\begin{equation*}
(\rho _{\alpha }(\gamma )f)(z)=\chi (\sigma (\gamma ))^{\alpha +1}\left( 
\frac{1}{-c_{\gamma }z+a_{\gamma }}\right) ^{\alpha +1}f\left( \frac{%
d_{\gamma }z-b_{\gamma }}{-c_{\gamma }z+a_{\gamma }}\right) \text{.}
\end{equation*}%
Then, the character of the translation generator is%
\begin{equation*}
\chi 
\begin{pmatrix}
1 & 1 \\ 
0 & 1%
\end{pmatrix}%
=e^{\pi i/6}\text{,}
\end{equation*}%
thus $\chi (-I)=-1$, $\chi ^{12}=1$ and $\rho _{12}$ is \emph{a genu\'{\i}ne
unitary representation} of $\func{PSL}(2,\mathbb{Z})$. This will simplify
some calculations.

For $\gamma \in \func{PSL}(2,\mathbb{Z})$, the ordinary \emph{left
representation} is defined on $l^{2}(\Gamma )$ by 
\begin{equation*}
\left( \lambda \left( \gamma \right) c\right) \left( \eta \right) =c(\gamma
^{-1}\eta )\text{, \ \ \ \ \ \ \ \ }c\in l^{2}(\Gamma ),\text{ \ \ \ \ \ }%
\gamma ,\eta \in \Gamma \text{.}
\end{equation*}%
The canonical basis vector of $l^{2}(\Gamma )$ is%
\begin{equation*}
\delta _{\eta }(\gamma )=\left\{ 
\begin{array}{cc}
1, & \gamma =\eta \\ 
0, & \gamma \neq \eta%
\end{array}%
\right. \text{, \ \ \ \ \ }\gamma ,\eta \in \Gamma \text{.}
\end{equation*}%
which is orthogonal in the sense 
\begin{equation}
\left\langle \delta _{\gamma },\delta _{\eta }\right\rangle =\delta _{\gamma
,\eta }\text{.}  \label{orthogonal}
\end{equation}%
The action of the ordinary left representationt is $\lambda \left( \gamma
\right) \delta _{\eta }=\delta _{\gamma \eta }$. In particular,%
\begin{equation*}
\lambda \left( \gamma \right) \delta _{e}=\delta _{\gamma }\text{.}
\end{equation*}%
A von Neumann algebra $M$ is a $\ast $-closed unital algebra of bounded
operators on a (complex) Hilbert space $\mathcal{H}$ which is closed under
the topology of pointwise convergence on $\mathcal{H}$. By the von Neumann
double-commutant theorem, $S^{^{\prime \prime }}$ (the double-commutant of $%
S $) is the smallest von Neumann algebra containing $S$. The ordinary left
representation of $\func{PSL}(2,\mathbb{Z})$ defines a group von Neumann
algebra $M$ on $l^{2}(\Gamma )$\ by%
\begin{equation*}
M=L(\Gamma )=\{\lambda \left( \gamma \right) :\gamma \in \Gamma \}^{^{\prime
\prime }}\subset B\left( l^{2}(\Gamma )\right) \text{,}
\end{equation*}%
while the representation $\rho _{\alpha }$ defines a von Neumann algebra on $%
\mathcal{A}_{\alpha }$%
\begin{equation*}
M_{\alpha }=\{\rho _{\alpha }\left( \gamma \right) :\gamma \in \Gamma
\}^{^{\prime \prime }}\subset B\left( \mathcal{A}_{\alpha }\right) \text{.}
\end{equation*}

\begin{definition}
A function $\Phi \in \mathcal{A}_{\alpha }$ is said to be wandering if%
\begin{equation*}
\left\langle \rho _{\alpha }(\gamma )\Phi ,\rho _{\alpha }(\eta )\Phi
\right\rangle =\delta _{\eta ,\gamma }\qquad \gamma ,\eta \in \Gamma \text{,}
\end{equation*}%
while $\Phi \in \mathcal{A}_{\alpha }$ is called a tracelike vector if 
\begin{equation*}
\sum_{\gamma \in \Gamma }\left\vert (\rho _{\alpha }(\gamma )\Phi
)(z)\right\vert ^{2}=K_{\alpha }(z,z)=\left( \frac{1}{2\func{Im}z}\right)
^{\alpha +1}\text{.}
\end{equation*}%
The two conditions hold when $\{\rho _{\alpha }\left( \gamma \right) \Phi
:\gamma \in \Gamma \}$ provides an orthonormal basis for $\mathcal{A}%
_{\alpha }$.
\end{definition}

The Fuchsian group $\func{PSL}(2,\mathbb{Z})$ defines a von Neumann algebra $%
M_{\alpha }$\ on $\mathcal{A}_{\alpha }$. Since 
\begin{equation*}
\func{covol}(\func{PSL}(2,\mathbb{Z}))=\frac{\pi }{3}\text{,}
\end{equation*}%
the corresponding module dimension is (\cite{RadulescuDimension},\cite[%
Theorem 4.8]{Jones}): 
\begin{equation}
\dim _{M_{\alpha }}\mathcal{A}_{\alpha }=\frac{\alpha \func{covol}(\func{PSL}%
(2,\mathbb{Z}))}{4\pi }=\frac{\alpha }{12}\text{.}  \label{dimensionformula}
\end{equation}%
Thus, $\dim _{M_{12}}\mathcal{A}_{12}=1$.\ By \cite[Theorem 1.1]{Jones} (or
by \cite{Seip}, since $\func{covol}(\func{PSL}(2,\mathbb{Z}))=\frac{\pi }{3}$
shows us that $\alpha =12$ corresponds to the Nyquist density), if $\frac{%
\alpha }{12}\leq 1$ (in particular when $\alpha =12$), then 
\begin{equation}
\overline{\func{span}}\{\rho _{\alpha }(\gamma )e_{0}^{\alpha }:\gamma \in
\Gamma \}=\mathcal{A}_{\alpha }\text{.}  \label{denseorbit}
\end{equation}%
The purpose of the paper is to specify, at $\alpha =12$, a particular unit
vector $\Phi \in \mathcal{A}_{12}$ whose $\func{PSL}(2,\mathbb{Z})$-orbit is
a complete orthonormal basis of $\mathcal{A}_{12}$ (we will often use $%
\alpha $ instead of $12$, exclusively for typographical reasons). This
provides a computable construction of a tracelike and wandering vector,
whose existence is assured by the following result.

\begin{theorem}
\label{Jones}\textbf{\ }\cite[Theorem 9.7]{Jones}. If $\alpha =12$, there
exists $\Phi \in \mathcal{A}_{\alpha }$ such that $\Phi $\ is a tracelike
and wandering vector.
\end{theorem}

It is observed in \cite[Pg. 7]{Jones} that Theorem \ref{Jones} shows that:%
\begin{equation*}
\end{equation*}

`\emph{\ ....there is a \textquotedblleft cyclic and separating trace
vector\textquotedblright\ for }$M_{\alpha }$\emph{\ in the Hilbert space,
and hence an anti-isomorphism between }$M_{\alpha }$\emph{\ and its
commutant on }$A_{\alpha }$\emph{. Now, in \cite{Radulescu}, Radulescu has
shown that the commutant }$M_{\alpha }^{\prime }$\emph{\ is always generated
in some sense by cusp forms which thus give a model for }$M_{\alpha
}^{\prime }$\emph{. In addition, Voiculescu in \cite{Voiculescu} has shown
that, at least for groups like }$\func{PSL}(2,\mathbb{Z})$\emph{, }$%
M_{\alpha }$\emph{\ has a random matrix model. Thus there exists a random
matrix model for cusp forms. This is a theorem, but it is of little use
unless one can lay one's hands on an explicit and manageable cyclic and
separating trace vector with which to implement the anti-isomorphism with
the commutant. If one did have such a vector one might be able to prove some
of the numerically well established relations between random matrices and
modular forms; see \cite{Keating}}.'%
\begin{equation*}
\end{equation*}

Constructing such a vector is stated as Problem 1 in \cite[pg. 8]{Jones}
(see also \cite{Ittersum}), where it is also shown that a tracelike and
wandering vector provides a cyclic and separating vector for the von Neumann
algebra. It is the purpose of this paper to construct this `magic function' $%
\Phi \in \mathcal{A}_{12}$ as a genu\'{\i}nely computable vector: first the
function $\Phi $ is constructed in Theorem \ref{main}; then, in Theorem \ref%
{finitecertificate}, $\Phi $ is approximated by finite sums $\Xi _{N}$, with
every finite approximation constructed using exact algebraic arithmetic. We
provide a finite computable error certificate for the approximation. To
obtain this certificate, the precise value $\alpha =12$ will be used. This
allows to compute all constants in $\mathbb{Q}\left( \zeta _{12}\right) =%
\mathbb{Q}\left( \sqrt{3},i\right) $, the algebraic field generated over $%
\mathbb{Q}$ by the $12th$-rooth of unity%
\begin{equation*}
\zeta _{12}=e^{2\pi i/12}=e^{\pi i/6}=\frac{\sqrt{3}+i}{2}\text{.}
\end{equation*}

\begin{figure}[!b]
\centering
\includegraphics[width=0.72\textwidth]{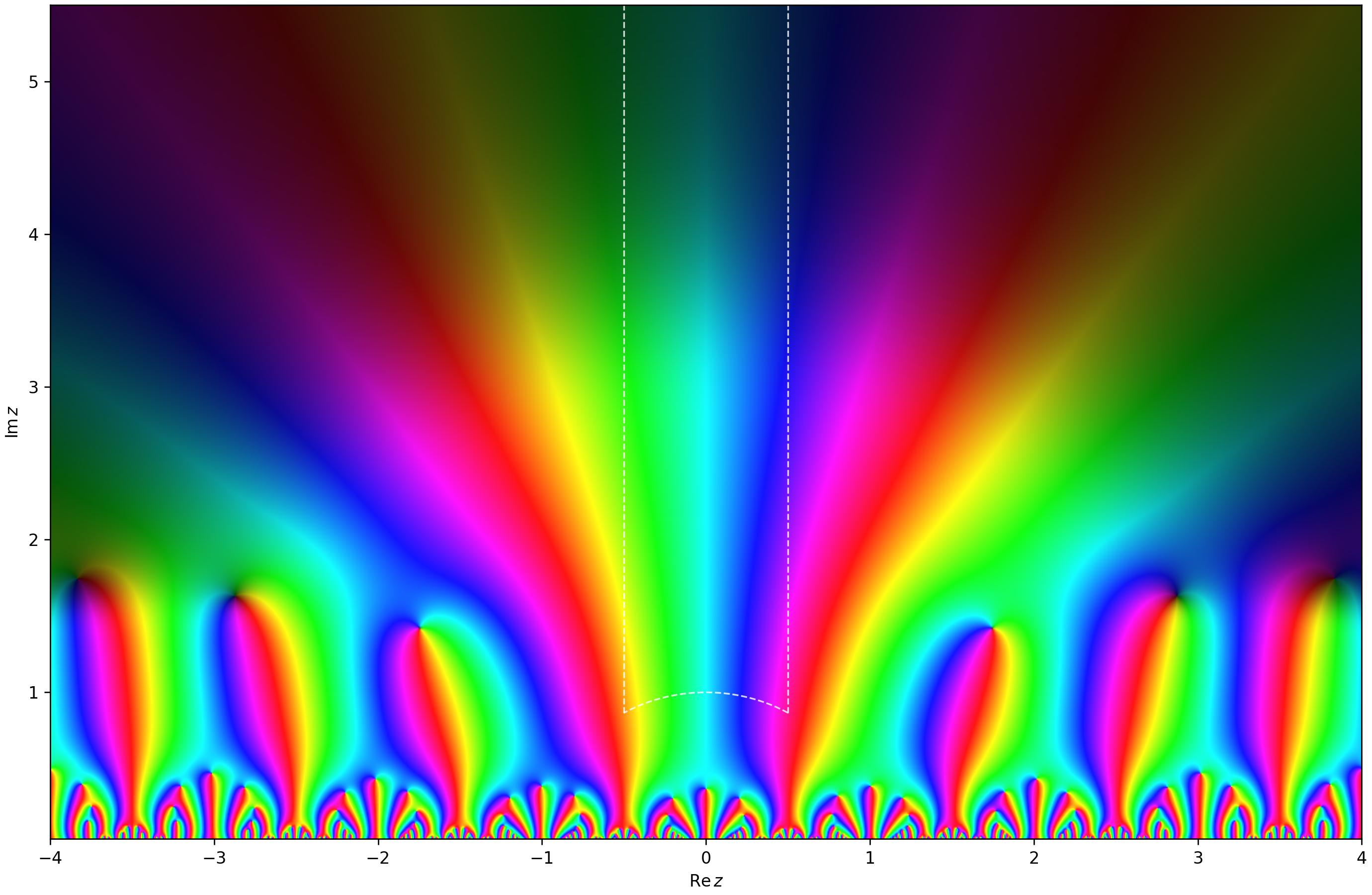}
\caption{The magic function in the upper half-plane, represented by
$\Phi _{m}=U_{m}\delta _{e}$ for $m=12$, with $%
|\mathcal{F}_{m}|=730$. Colors represent $\arg \Phi _{m}(z)$ and brightness
represents $|\Phi _{m}(z)|$. The dashed curve indicates the standard modular
fundamental domain. Reminiscences of modularity
can be observed, despite $\Phi _{m}$ not being automorphic.}
\label{MagicFunctionPlot}
\end{figure}

The computable function $\Phi $ is wandering, and a trace, cyclic and
separating vector for $M_{12}$. In Figure 1, one can recognize some global
modular features (the patterns inside the fundamental domain are replicated,
up to some distortion) combined with those of the Bergman kernel. This
happens because its construction can be informally seen as a mechanism of
(almost) modularization and orthogonalization of the Bergman kernel. Once $%
\Phi $ is obtained, we can define the operator (see \cite[Lemma 3.2.2]%
{GoodmanDeLaHarpeJones})%
\begin{equation*}
J_{\Phi }\rho _{12}(\gamma )\Phi =\rho _{12}(\gamma ^{-1})\Phi 
\end{equation*}%
and extend it conjugate-linearly. Then $J_{\Phi }$ satisfies $J_{\Phi
}M_{12}J_{\Phi }=M_{12}^{\prime }$. Thus, the map%
\begin{equation*}
\beta (\varkappa )=J_{\Phi }\varkappa ^{\ast }J_{\Phi }
\end{equation*}%
defines a linear $\ast $-anti-isomorphism%
\begin{equation*}
\beta :M_{12}\rightarrow M_{12}^{\prime }\text{.}
\end{equation*}

We observe that R\u{a}dulescu's results from \cite{Radulescu} have been
recently extended to general Fuchsian groups \cite{Yang}. This paper
strictly concerns the case $\alpha =12$, but the construction in Theorem \ref%
{main} extends to general Fuchsian groups. However, extending Theorem \ref%
{finitecertificate} is more delicate, since we have used ingredients which
are automatically granted for $\func{PSL}(2,\mathbb{Z})$, such as a
computable matrix realization and computability of the Gramian entries with
certified precision, among other properties. Thus, we are not in a position
of making claims in this direction, but believe that, at least for some
groups, the method may work with some adaptations, and that the problem is
worth being pursued.

A much more difficult problem arises if one considers the action of $\func{%
PSL}(2,\mathbb{Z})$ or other Fuchsian groups in non-holomorphic spaces, as
it naturally occurs for Maass forms \cite{Maass}. A first step in this
direction was taken in \cite{Affine}, where Jones dimension formula was
extended to algebras generated by an infinite family of projective
representations acting in non-holomorphic spaces, of which the discrete
eigenspaces of the Maass Laplacian are a special case \cite[pg. 8]{Affine}.
As a result of the von Neumann dimension formula proved in \cite{Affine},
Theorem \ref{Jones} extends to this setting (just mimic the proof of \cite[%
Theorem 9.7]{Jones}). Thus, there is a wandering and tracelike vector for
the whole infinite family of non-holomorphic spaces considered in \cite%
{Affine}. \textbf{Problem: }\emph{find one}. In this case, an attempt to
generalize our methods will face significant obstructions. For instance,
Jones' \cite[Theorem 1.1]{Jones} and Seip's \cite{Seip} sufficient
conditions for completeness, required in our construction, completely depend
on the holomorphy of $\mathcal{A}_{\alpha }$.

\subsection{Outline of the paper}

The paper is organized as follows. The next section gathers some preparation
results, including the restriction of $\Gamma $ to finite subsets of
increasing `radius', which will be used in the construction of $\Phi $. The
construction itself is made on Section 3 and the effective procedure for its
computation is described in Section 4.

\section{Preliminary results and proof strategy}

\subsection{A sequence of finite subsets $\mathcal{F}_{m}\subset \Gamma $}

From this point through Section \ref{effectiveSection}, assume the critical
value $\alpha =12$, while retaining $\alpha $ in many of the formulas for
typographical reasons. For 
\begin{equation*}
M=%
\begin{pmatrix}
a & b \\ 
c & d%
\end{pmatrix}%
\in \limfunc{SL}(2,\mathbb{Z})
\end{equation*}%
put 
\begin{equation*}
R(M)=\max \{|a|,|b|,|c|,|d|\}\text{.}
\end{equation*}%
Since $R(M)=R(-M)$, this descends to a well defined function on $\Gamma =%
\func{PSL}(2,\mathbb{Z})$, namely 
\begin{equation*}
R(\gamma )=R(M)\text{.}
\end{equation*}%
For $m\geq 1$, let 
\begin{equation}
\mathcal{F}_{m}=\{\gamma \in \Gamma :R(\gamma )\leq m\}\text{.}  \label{Fm}
\end{equation}%
The sets $\mathcal{F}_{m}$ are finite, symmetric and increasing, and their
union is $\Gamma $.

\subsection{The Gram operator}

Given $\gamma ,\eta \in \Gamma $, write 
\begin{equation}
\sigma (\gamma ^{-1}\eta )=%
\begin{pmatrix}
a_{\gamma ^{-1}\eta } & b_{\gamma ^{-1}\eta } \\ 
c_{\gamma ^{-1}\eta } & d_{\gamma ^{-1}\eta }%
\end{pmatrix}%
\text{.}  \label{relativeproduct}
\end{equation}%
Thus the matrix in (\ref{relativeproduct}) is the canonical representative
of the relative group element $\gamma ^{-1}\eta $. The \emph{Gram operator}
is the operator $G$, whose entries are given by%
\begin{equation}
\left( G\right) _{\gamma ,\eta }=\left\langle \rho _{\alpha }(\gamma
)e_{0}^{\alpha },\rho _{\alpha }(\eta )e_{0}^{\alpha }\right\rangle _{%
\mathcal{A}_{\alpha }},\qquad \gamma ,\eta \in \Gamma \text{.}
\label{GramianEntries}
\end{equation}%
\ Our construction will use, in a first step, the element $n=0$ of (\ref%
{basis}): 
\begin{equation*}
e_{0}^{\alpha }(z)=\left( \frac{2i}{z+2i}\right) ^{\alpha +1}\text{.}
\end{equation*}%
The action of 
\begin{equation*}
\sigma (\gamma )=%
\begin{pmatrix}
a_{\gamma } & b_{\gamma } \\ 
c_{\gamma } & d_{\gamma }%
\end{pmatrix}%
\text{,}
\end{equation*}%
on $e_{0}^{\alpha }(z)$ is explicitly given as 
\begin{equation}
\left( \rho _{\alpha }(\gamma )e_{0}^{\alpha }\right) (z)=\chi (\sigma
(\gamma ))^{\alpha +1}\left( \frac{2i}{(d_{\gamma }-2ic_{\gamma
})z+2ia_{\gamma }-b_{\gamma }}\right) ^{\alpha +1}\text{.}
\label{orbitHalfPlane}
\end{equation}%
At $\alpha =12$, one has $\chi (\sigma (\gamma ))^{\alpha +1}=\chi (\sigma
(\gamma ))$. In order to obtain a bound for the Gram operator, the
coefficients $\left\langle e_{0}^{\alpha },\rho _{\alpha }(\gamma
)e_{0}^{\alpha }\right\rangle $ need to be estimated. First we show that $%
\left\langle e_{0}^{\alpha },\rho _{\alpha }(\gamma )e_{0}^{\alpha
}\right\rangle =2^{\alpha +1}\left( \rho _{\alpha }(\gamma )e_{0}^{\alpha
}\right) \left( 2i\right) $: set $F=\rho _{\alpha }(\gamma )e_{0}^{\alpha }$
in $\left\langle e_{0}^{\alpha },F\right\rangle =\left\langle K_{\alpha
}(.,2i),F\right\rangle =2^{\alpha +1}F(2i)$, after observing that $K_{\alpha
}(z,2i)=2^{-(\alpha +1)}e_{0}^{\alpha }(z)$. Together with (\ref%
{orbitHalfPlane}), this leads to the identity 
\begin{equation*}
\left\vert \left\langle e_{0}^{\alpha },\rho _{\alpha }(\gamma
)e_{0}^{\alpha }\right\rangle \right\vert =\left\vert \frac{4}{2(a_{\gamma
}+d_{\gamma })-i(4c_{\gamma }-b_{\gamma })}\right\vert ^{\alpha +1}=\left( 
\frac{16}{Q_{\sigma (\gamma )}}\right) ^{1/2}\left( \frac{16}{Q_{\sigma
(\gamma )}}\right) ^{\alpha /2}\text{,}
\end{equation*}%
with $Q_{\sigma (\gamma )}=$\ $4(a_{\gamma }+d_{\gamma })^{2}+(4c_{\gamma
}-b_{\gamma })^{2}$. Since $Q_{\sigma (\gamma )}\geq 16$ and $Q_{\sigma
(\gamma )}\geq \max \{|a_{\gamma }|^{2},|b_{\gamma }|^{2},|c_{\gamma
}|^{2},|d_{\gamma }|^{2}\}=R(\sigma (\gamma ))^{2}$, we have 
\begin{equation}
\left\vert \langle e_{0}^{\alpha },\rho _{\alpha }(\gamma )e_{0}^{\alpha
}\rangle \right\vert \leq 4^{\alpha }R(\sigma (\gamma ))^{-\alpha }\text{.}
\label{Gramdecay}
\end{equation}%
There are at most $(2r+1)^{4}-(2r-1)^{4}\leq 80r^{3}$ integer matrices whose
maximum entry modulus equals $r$. Hence, for $\alpha =12$ and every integer $%
M\geq 1$, 
\begin{equation}
\sum_{\gamma \in \Gamma }\left\vert \langle e_{0}^{\alpha },\rho _{\alpha
}(\gamma )e_{0}^{\alpha }\rangle \right\vert \leq 4^{\alpha
}80\sum_{r=1}^{\infty }r^{3-\alpha }\leq 4^{\alpha }80\frac{9}{8}\leq
4^{12}90\text{.}  \label{tailbound}
\end{equation}%
Since $\left\vert (G)_{\gamma ,\eta }\right\vert =\left\vert \langle \rho
_{\alpha }(\gamma ^{-1}\eta )e_{0}^{\alpha },e_{0}^{\alpha }\rangle
\right\vert $, 
\begin{equation*}
\sum_{\gamma \in \Gamma }\left\vert (G)_{\gamma ,\eta }\right\vert \leq
4^{12}90\text{, \ \ \ }\eta \in \Gamma \text{.}
\end{equation*}%
The same inequality holds with the roles of $\gamma $ and $\eta $\ reversed.
Thus, we can first define $G$ for $c$ supported in $\mathcal{F}_{m}$\ and
Schur's test gives, writing $L=4^{12}100$,%
\begin{equation}
\Vert Gc\Vert _{\ell ^{2}(\Gamma )}\leq 4^{12}90\Vert c\Vert _{\ell
^{2}(\Gamma )}<L\Vert c\Vert _{\ell ^{2}(\Gamma )}\text{,}
\label{SchurGramian}
\end{equation}%
so that $G$ extends to a bounded operator $G:\ell ^{2}(\Gamma )\rightarrow
\ell ^{2}(\Gamma )$.

\subsection{The synthesis operator}

The \emph{synthesis operator} of the complete sequence $\{\rho _{\alpha
}(\gamma )e_{0}^{\alpha }\}_{\gamma \in \Gamma }$ is first defined for $c$
supported in $\mathcal{F}_{m}$\ as 
\begin{equation}
Ac=\sum_{\gamma \in \Gamma }c_{\gamma }\rho _{\alpha }(\gamma )e_{0}^{\alpha
}\text{.}  \label{synthesis}
\end{equation}%
For such finitely supported $c$, (\ref{SchurGramian}) gives%
\begin{equation*}
\Vert Ac\Vert ^{2}=\left\langle c,Gc\right\rangle \leq 4^{12}90\Vert c\Vert
^{2}\text{,}
\end{equation*}%
so $A$ extends to a bounded operator $A:\ell ^{2}(\Gamma )\rightarrow 
\mathcal{A}_{\alpha }$ , with $G=A^{\ast }A$. Then 
\begin{equation}
0\leq G<L\,I,\qquad \Vert A\Vert ^{2}=\Vert G\Vert <\sqrt{L}=40960\text{.}
\label{operatorbound}
\end{equation}%
We can already see why we have chosen $L=4^{12}100$ in \ref{SchurGramian}\
it bounds $\Vert A\Vert $ by the integer $\sqrt{L}=40960$, so that $L^{-1/2}$
is a rational number. This will be essential in the last section. The
following Lemma gathers the required general results about $A$, $U$ and $G$.

\begin{lemma}
\label{injective} Assume $\alpha =12$. The synthesis operator $A:\ell
^{2}(\Gamma )\rightarrow \mathcal{A}_{\alpha }$ is injective and has dense
range. In the polar decomposition 
\begin{equation}
A=U\left\vert A\right\vert \text{, \ \ }G=A^{\ast }A\text{,}  \label{polar}
\end{equation}%
the polar factor $U:l^{2}(\Gamma )\rightarrow \mathcal{A}_{\alpha }$ is
unitary. Moreover, the restriction 
\begin{equation}
U:\left( \ker A\right) ^{\perp }\rightarrow \mathcal{A}_{\alpha }
\label{restriction}
\end{equation}%
is an unitary module map between the corresponding $M$-modules.
\end{lemma}

\noindent \textit{Proof.}\ By (\ref{denseorbit}), $\overline{ranA}=\mathcal{A%
}_{12}$. Since $\left\vert A\right\vert =\left( A^{\ast }A\right)
^{1/2}=G^{1/2}$, the polar decomposition of $A$ is 
\begin{equation*}
A=U\left\vert A\right\vert =UG^{1/2}\text{,}
\end{equation*}%
where $U$ is the partial isometric polar factor, defined as a unique partial
isometry (it acts isometrically on one closed subspace and vanishes on its
orthogonal complement), whose initial and final spaces are 
\begin{equation*}
\overline{ran\left\vert A\right\vert }=\left( \ker A\right) ^{\perp }\text{,
\ \ \ }\overline{ranA}=\mathcal{A}_{\alpha }\text{,}
\end{equation*}%
respectively. The action on the canonical basis vector $\delta _{\eta }$ is $%
A\delta _{\eta }=\rho _{\alpha }(\eta )e_{0}^{\alpha }$. Combining this with 
$A\lambda (\gamma )\delta _{\eta }=A\delta _{\gamma \eta }$, we see that the
synthesis\ operator intertwines the ordinary left regular representation of $%
\Gamma $ with the representation $\rho _{\alpha }$: 
\begin{equation}
A\lambda (\gamma )=\rho _{\alpha }(\gamma )A\text{.}  \label{intertwining}
\end{equation}%
Consequently, $\lambda (\gamma )^{\ast }G\lambda (\gamma )=A^{\ast }A=G$.
Thus, $G$ and, consequently, $\left\vert A\right\vert =G^{1/2}$, commute
with $\lambda (\gamma )$. It follows from $A=U\left\vert A\right\vert $ and (%
\ref{intertwining}) that $\left( U\lambda (\gamma )-\rho _{\alpha }(\gamma
)U\right) \left\vert A\right\vert =0$. Since $ran\left\vert A\right\vert $
is dense in $\left( \ker A\right) ^{\perp }$ and $\ker A$ is invariant under 
$\lambda (\gamma )$, this implies%
\begin{equation}
U\lambda (\gamma )=\rho _{\alpha }(\gamma )U\text{, \ \ \ }\gamma \in \Gamma 
\text{.}  \label{polarintertwining}
\end{equation}%
Thus, the restriction (\ref{restriction}) is an unitary module map between $%
\left( \ker A\right) ^{\perp }$ and $\mathcal{A}_{\alpha }$\ as $M$-modules.
Unitary invariance of von Neumann dimension \cite[Proposition 3.2.4 a)]%
{GoodmanDeLaHarpeJones}, implies that $\dim _{M}\left( \ker A\right) ^{\perp
}=\dim _{M_{\alpha }}\mathcal{A}_{\alpha }$. Since we know from (\ref%
{dimensionformula}) that $\dim _{M_{\alpha }}\mathcal{A}_{\alpha }=1$, when $%
\alpha =12$, $\dim _{M}\left( \ker A\right) ^{\perp }=\dim _{M_{\alpha }}%
\mathcal{A}_{\alpha }=1$. Thus, $\dim _{M}(l^{2}(\Gamma ))=1$ and additivity
of von Neumann dimension gives 
\begin{equation*}
1=\dim _{M}(\ker A)+\dim _{M}\left( \ker A\right) ^{\perp }=\dim _{M}(\ker
A)+1
\end{equation*}%
and we conclude that $\dim _{M}(\ker A)=0$. Since von Neumann dimension is
faithful on closed Hilbert $M$-submodules, this implies $\ker A=\{0\}$ (the
required dimension properties are recalled in \cite[Section 3]{Jones}). Thus
the initial and final spaces of the polar factor are the whole source and
target, so $U:l^{2}(\Gamma )\rightarrow \mathcal{A}_{\alpha }$ is unitary at 
$\alpha =12$. Uniqueness of polar decomposition and (\ref{intertwining})
show that $U$ intertwines the representations. \hfill $\square $

\subsection{The intuition behind the proofs}

The idea behind the proof has a simple intuition. Once we are able to define
the magic function by $\Phi =U\delta _{e}$, then%
\begin{equation*}
\left\langle \rho _{\alpha }(\gamma )\Phi ,\rho _{\alpha }(\eta )\Phi
\right\rangle =\left\langle U\delta _{\gamma },U\delta _{\eta }\right\rangle
=\left\langle \delta _{\gamma },\delta _{\eta }\right\rangle =\delta
_{\gamma ,\eta }\text{.}
\end{equation*}%
Thus, the problem can be solved if we find a way of determining $U$ from $%
A=UG^{1/2}$. This becomes non-trivial because we don't know if the Gramian $G
$ of the complete sequence $\{\rho _{\alpha }(\gamma )e_{0}^{\alpha
}\}_{\gamma \in \Gamma }$, is invertible. To overcome this difficulty, we
use the finite subsets $\mathcal{F}_{m}\subset \Gamma $ (\ref{Fm}). These
sets can informally be seen as `circles' containing the subsets of $\Gamma $%
\ up to a prescribed `radius' depending on $m$, which will later made large.
The synthesis and Gramian operators of $\{\rho _{\alpha }(\gamma
)e_{0}^{\alpha }\}_{\gamma \in \Gamma }$ are then restricted to $\mathcal{F}%
_{m}$\ and the resulting finite-dimensional operators produce a sequence of
Gram matrices $G_{m}$ and polar partial isometries $U_{m}$, such that $%
U_{m}\rightarrow U$ strongly, providing therefore the sought function $\Phi $
as the local uniform limit of $U_{m}\delta _{e}$. This provides the main
result of the next section. In the last section of the paper, we truncate
the binomial formula to obtain a polynomial approximation to $t^{-1/2}$ and
turn this approximation into a sequence of polynomial vectors $P_{N}$
converging in $\mathcal{A}_{\alpha }$ to $\Phi =U\delta _{e}$. Then we
introduce a sequence of functions $\Xi _{N}$ involving only finite
computable matrices, such that $\left\Vert \Xi _{N}-P_{N}\right\Vert _{%
\mathcal{A}_{\alpha }}<2^{-N}$. The construction culminates in showing that $%
\left\Vert \Phi -\Xi _{N}\right\Vert \rightarrow 0$\ effectively, (with a
computable bound). This leads to an \emph{efective procedure} \cite%
{Computability}, therefore demonstrating computability of $\Phi $.

\section{Construction of $\Phi $}

\subsection{Finite Gram matrices and the main result}

Define \emph{the finite Gram matrix} $G_{m}$, indexed by $\mathcal{F}_{m}$,
by%
\begin{equation}
\left( G_{m}\right) _{\gamma ,\eta }=\left\langle \rho _{\alpha }(\gamma
)e_{0}^{\alpha },\rho _{\alpha }(\eta )e_{0}^{\alpha }\right\rangle _{%
\mathcal{A}_{\alpha }},\qquad \gamma ,\eta \in \mathcal{F}_{m}.
\label{gramFinite}
\end{equation}%
Writing the relative representative as in (\ref{relativeproduct}), using (%
\ref{orbitHalfPlane}) and $\chi ^{12}=1$, its entries at $\alpha =12$ are
explicitly given by 
\begin{equation}
(G_{m})_{\gamma ,\eta }=\chi \left( \sigma (\gamma ^{-1}\eta )\right) \left( 
\frac{4}{2(a_{\gamma ^{-1}\eta }+d_{\gamma ^{-1}\eta })-i(4c_{\gamma
^{-1}\eta }-b_{\gamma ^{-1}\eta })}\right) ^{13},\qquad \gamma ,\eta \in 
\mathcal{F}_{m}.  \label{explicitCriticalGram}
\end{equation}%
In particular, 
\begin{equation*}
(G_{m})_{\gamma ,\gamma }=1,\qquad (G_{m})_{\gamma ,\eta }=\overline{%
(G_{m})_{\eta ,\gamma }}.
\end{equation*}

\begin{theorem}
\label{main} Assume $\alpha =12$. Every $G_{m}$ is positive definite. Let 
\begin{equation*}
v_{m}=G_{m}^{-1/2}\delta _{e}
\end{equation*}%
and 
\begin{equation}
\Phi _{m}(z):=\sum_{\gamma \in \mathcal{F}_{m}}v_{m}(\gamma )\left( \rho
_{\alpha }(\gamma )e_{0}^{\alpha }\right) \left( z\right) \text{.}
\label{Phim}
\end{equation}%
There exists a unit vector $\Phi \in \mathcal{A}_{\alpha }$ such that 
\begin{equation}
\Phi _{m}\longrightarrow \Phi \quad \text{in }\mathcal{A}_{\alpha }\quad 
\text{locally uniformly on }\mathbb{H}  \label{Philimit}
\end{equation}%
and such that its modular orbit is an orthonormal basis of $\mathcal{A}%
_{\alpha }$.
\end{theorem}

\subsection{\textit{Proof of Theorem \protect\ref{main}.}}

If $c\neq 0$ is supported on $\mathcal{F}_{m}$, then%
\begin{equation*}
U(G^{1/2}c)=Ac=\sum_{\gamma \in \mathcal{F}_{m}}c(\gamma )\left( \rho
_{\alpha }(\gamma )e_{0}^{\alpha }\right) \text{.}
\end{equation*}%
By Lemma \ref{injective}, $A:\ell ^{2}(\Gamma )\rightarrow \mathcal{A}%
_{\alpha }$ is injective, thus $Ac\neq 0$ and 
\begin{equation*}
0<\Vert Ac\Vert _{\mathcal{A}_{\alpha }}^{2}=\Vert A_{m}c\Vert _{\mathcal{A}%
_{\alpha }}^{2}=c^{\ast }G_{m}c\text{,}
\end{equation*}%
where $A_{m}=AP_{m}$ and $P_{m}$ is the coordinate projection onto $\ell
^{2}(\mathcal{F}_{m})$. We conclude that $G_{m}$ is positive definite.
Extend the polar partial isometry $U_{m}$ of $A_{m}$ by zero on its initial
kernel. Now, since $e\in \mathcal{F}_{m}$, we can define $\Phi _{m}$ as 
\begin{equation*}
\Phi _{m}=U_{m}\delta _{e}=A(G_{m}^{-1/2}\delta _{e})\text{.}
\end{equation*}%
Thus, 
\begin{equation*}
\Phi _{m}(z)=\sum_{\gamma \in \mathcal{F}_{m}}v_{m}(\gamma )\left( \rho
_{\alpha }(\gamma )e_{0}^{\alpha }\right) (z)\text{,}
\end{equation*}%
where $v_{m}(\gamma )=(G_{m}^{-1/2})_{\gamma ,e}$. Both $A_{m}\rightarrow A$
and $A_{m}^{\ast }\rightarrow A^{\ast }$ strongly and are uniformly bounded.
Since $P_{m}\rightarrow I$ strongly, $P_{m}GP_{m}\rightarrow G$, also
strongly. By (\ref{SchurGramian}), the spectrum of both $P_{m}GP_{m}$ and $G$
are contained in $[0,L]$. Hence, uniform Weierstrass polynomial
approximation of the square-root function on $[0,L]$ yields, by functional
calculus, 
\begin{equation*}
\left\vert A_{m}\right\vert =(P_{m}GP_{m})^{1/2}\longrightarrow
G^{1/2}=\left\vert A\right\vert \quad \text{strongly.}
\end{equation*}%
For $x\in \ell ^{2}(\Gamma )$, 
\begin{equation*}
(U_{m}-U)\left\vert A\right\vert x=U_{m}(\left\vert A\right\vert -\left\vert
A_{m}\right\vert )x+(A_{m}-A)x\longrightarrow 0\text{.}
\end{equation*}%
The range of $\left\vert A\right\vert $ is dense because $A$ is injective,
while $\left\Vert U_{m}\right\Vert \leq 1$. Therefore $U_{m}\rightarrow U$
strongly, and we obtain 
\begin{equation*}
\Phi _{m}\longrightarrow \Phi :=U\delta _{e}\quad \text{in }\mathcal{A}%
_{\alpha }\text{.}
\end{equation*}%
Since the orbit of $\delta _{e}$ is the standard orthonormal basis,
unitarity gives 
\begin{equation*}
\left\langle \rho _{\alpha }(\gamma )\Phi ,\rho _{\alpha }(\eta )\Phi
\right\rangle =\left\langle U\delta _{\gamma },U\delta _{\eta }\right\rangle
=\left\langle \delta _{\gamma },\delta _{\eta }\right\rangle =\delta
_{\gamma ,\eta }
\end{equation*}%
and since we already know that $\{\rho _{\alpha }(\gamma )\Phi :\gamma \in
\Gamma \}$ is complete, we conclude that it is an orthonormal basis for $%
\mathcal{A}_{\alpha }$.

\section{Effective computation of $\Phi $}

\label{effectiveSection}

We want to approximate $\Phi =U\delta _{e}$. In the first step we
approximate the square root $t^{-1/2}$ by turning, via functional calculus,
the truncated binomial polynomial $p_{N}(t)$ into the sequence (\ref{etaN})
converging in $\mathcal{A}_{\alpha }$ to $\Phi =U\delta _{e}$. The main
computability theorem is then proved by approximating this sequence by
computable finite matrices, and bounding the approximation by a computable
constant.

Let $L=4^{12}100$. Setting $x=1-t/L$ in the binomial formula,%
\begin{equation*}
(1-x)^{-\frac{1}{2}}=\sum_{j=0}^{\infty }\binom{2j}{j}4^{-j}x^{j},\text{ \ \
\ \ \ }\left\vert x\right\vert <1\text{.}
\end{equation*}%
gives 
\begin{equation*}
t^{-1/2}=L^{-1/2}\sum_{j=0}^{\infty }\binom{2j}{j}4^{-j}\left( 1-\frac{t}{L}%
\right) ^{j},\text{ \ \ }t\in (0,L]\text{.}
\end{equation*}%
This suggests defining the polynomial 
\begin{equation}
p_{N}(t)=L^{-1/2}\sum_{j=0}^{N}\binom{2j}{j}4^{-j}\left( 1-\frac{t}{L}%
\right) ^{j}\text{,}  \label{polynomialEffective}
\end{equation}%
and the vector 
\begin{equation}
P_{N}=Ap_{N}(G)\delta _{e}\text{.}  \label{etaN}
\end{equation}%
Thus $p_{N}$ is a scalar polynomial approximation to $t^{-1/2}$. This will
be used via the functional calculus of the Gram operator. As a preparation
for our main result, we need to understand the convergence $%
G^{1/2}p_{N}(G)\delta _{e}\rightarrow \delta _{e}$.

\begin{lemma}
\label{polynomialcertificate} The vectors $P_{N}$ converge in $\mathcal{A}%
_{\alpha }$ to $\Phi =U\delta _{e}$, and 
\begin{equation}
\Vert \Phi -P_{N}\Vert _{\mathcal{A}_{\alpha }}^{2}\leq 1-\Vert P_{N}\Vert _{%
\mathcal{A}_{\alpha }}^{2}.  \label{etacertificate}
\end{equation}
\end{lemma}

\noindent \textit{Proof. }Clearly, 
\begin{equation*}
0\leq \sqrt{t}\,p_{N}(t)\leq 1,\qquad \sqrt{t}\,p_{N}(t)\longrightarrow
\left\{ 
\begin{array}{cc}
1, & t>0 \\ 
0, & t=0%
\end{array}%
\right. .
\end{equation*}%
Since $G$ is injective, the spectral measure of $\delta _{e}$ for $G$ has no
mass at zero and $\nu _{\delta _{e}}\left( \left[ 0,L\right] \right) =1$.
Thus, since $0<t\leq L$, when $N\rightarrow \infty $, 
\begin{equation*}
\left\Vert \left( G^{1/2}p_{N}(G)-I\right) \delta _{e}\right\Vert
^{2}=\int_{[0,L]}\left\vert \sqrt{t}\,p_{N}(t)-1\right\vert ^{2}d\nu
_{\delta _{e}}\rightarrow 0\text{.}
\end{equation*}%
Thus, the displayed norm goes to zero in dominated convergence and\ since $%
P_{N}=UG^{1/2}p_{N}(G)\delta _{e}$, so does $\Vert \Phi -P_{N}\Vert _{%
\mathcal{A}_{\alpha }}^{2}$. Now,%
\begin{equation*}
\left\Vert P_{N}-\Phi \right\Vert ^{2}=\left\Vert \left(
UG^{1/2}p_{N}(G)-U\right) \delta _{e}\right\Vert ^{2}=\left\Vert \left(
G^{1/2}p_{N}(G)-I\right) \delta _{e}\right\Vert ^{2}\text{.}
\end{equation*}%
Finally, observing that $(1-u)^{2}\leq 1-u^{2}$ for $0\leq u\leq 1$ and
applying this pointwise to $u=\sqrt{\lambda }\,p_{N}(\lambda )$\ gives%
\begin{eqnarray*}
\left\Vert P_{N}-\Phi \right\Vert ^{2} &=&\int_{[0,L]}\left\vert \sqrt{t}%
\,p_{N}(t)-1\right\vert ^{2}d\nu _{\delta _{e}} \\
&\leq &1-\int_{[0,L]}\left\vert \sqrt{t}\,p_{N}(t)\right\vert ^{2}d\nu
_{\delta _{e}} \\
&=&1-\Vert UG^{1/2}p_{N}(G)\delta _{e}\Vert _{\mathcal{A}_{\alpha }}^{2} \\
&=&1-\Vert P_{N}\Vert _{\mathcal{A}_{\alpha }}^{2}\text{.}
\end{eqnarray*}%
\hfill $\square $

The main theorem will now be proved by approximation the vectors $P_{N}$ by
computable finite matrices, and bounding the approximation by a computable
constant. For $N\geq 1$, let $G^{(M_{N})}$ be the operator defined as%
\begin{equation*}
\left( G^{(M_{N})}\right) _{\gamma ,\eta }=%
\begin{cases}
\left( G\right) _{\gamma ,\eta }, & R\left( \sigma (\gamma ^{-1}\eta
)\right) \leq M_{N}, \\ 
0, & R\left( \sigma (\gamma ^{-1}\eta )\right) >M_{N},%
\end{cases}%
\qquad \gamma ,\eta \in \Gamma \text{,}
\end{equation*}%
where $\left( G\right) _{\gamma ,\eta }$\ is the Gram coefficient (\ref%
{GramianEntries}). Choose $M_{N}$ such that 
\begin{equation}
\text{\ }L\,M_{N}^{-8}\leq \frac{2^{-N}}{(N+1)^{2}}\text{.}  \label{cutoffs}
\end{equation}
We will replace $p_{N}(G)\delta _{e}$ by a vector $p_{N}\left( B_{N}\right)
\delta _{e}$, where the finite matrix $B_{N}$ is the compression of $%
G^{(M_{N})}$ by $P_{R_{N}}$\ on $\ell ^{2}(\mathcal{F}_{R_{N}})$, defined as%
\begin{equation*}
B_{N}=P_{R_{N}}G^{(M_{N})}P_{R_{N}}\text{,}
\end{equation*}%
with the much larger radius $R_{N}=(2M_{N})^{N}$\ chosen so that the
compression does not affect paths of length at most $N$ beginning at $e$.
The entries of $B_{N}$ are 
\begin{equation}
\left( B_{N}\right) _{\gamma ,\eta }=%
\begin{cases}
\left( G\right) _{\gamma ,\eta }, & R\left( \sigma (\gamma ^{-1}\eta
)\right) \leq M_{N}, \\ 
0, & R\left( \sigma (\gamma ^{-1}\eta )\right) >M_{N},%
\end{cases}%
\qquad \gamma ,\eta \in \mathcal{F}_{R_{N}}\text{.}  \label{BN}
\end{equation}%
Define 
\begin{equation}
q_{N}=p_{N}\left( B_{N}\right) \delta _{e}=L^{-1/2}\sum_{j=0}^{N}\binom{2j}{j%
}4^{-j}\left( I-\frac{B_{N}}{L}\right) ^{j}\delta _{e}  \label{qN}
\end{equation}%
and extend $q_{N}$ by zero outside $\mathcal{F}_{R_{N}}$. The next step is
to compare $P_{N}$ with%
\begin{equation}
\Xi _{N}(z):=Aq_{N}(z)=\sum_{\gamma \in \mathcal{F}_{R_{N}}}q_{N}(\gamma
)(\rho _{\alpha }(\gamma )e_{0}^{\alpha })(z)  \label{Xi}
\end{equation}%
At $\alpha =12$, all entries of $B_{N}$, $q_{N}$, and the scalar
coefficients of $\Xi _{N}$ belong to the field $\mathbb{Q}(\zeta _{12})=%
\mathbb{Q}(i,\sqrt{3})$. Hence every finite approximation can be constructed
using exact algebraic arithmetic.

For the next result, denote by $G_{R_{N}}$ be the full Gram matrix on $%
\mathcal{F}_{R_{N}}$, define 
\begin{equation}
\nu _{N}=(q_{N}^{\ast }G_{R_{N}}q_{N})^{1/2}\text{,}  \label{EN}
\end{equation}%
and observe that $\nu _{N}$ is is obtained from a finite matrix calculation.
The following is the main result of this section.

\begin{theorem}
\textbf{\ }\label{finitecertificate} The finite approximants satisfy 
\begin{equation}
\Vert \Xi _{N}-P_{N}\Vert _{\mathcal{A}_{\alpha }}\leq 2^{-N}\text{.}
\label{Xieta}
\end{equation}%
Moreover, 
\begin{equation}
\Vert \Phi -\Xi _{N}\Vert _{\mathcal{A}_{\alpha }}\leq 2^{-N}+\sqrt{1-\left(
\max \{0,\nu _{N}-2^{-N}\}\right) ^{2}}\longrightarrow 0\text{,}
\label{effective}
\end{equation}%
is a finite computable error certificate. Moreover, for $K\subset \mathbb{H}$
compact, 
\begin{equation}
\sup_{z\in K}\left\vert \Phi (z)-\Xi _{N}(z)\right\vert \leq \sup_{z\in K}%
\sqrt{K_{\alpha }(z,z)}\left\Vert \Phi -\Xi _{N}\right\Vert _{\mathcal{A}%
_{\alpha }}\rightarrow 0\text{.}  \label{pointwise}
\end{equation}
Therefore, the norm certificate induces an explicit pointwise, locally
uniform error certificate. As a result, $\Phi $ is effectively computable as
a holomorphic function in $\mathbb{H}$.
\end{theorem}

\noindent \textit{Proof.}\ Estimating the tail of the Gram operator as in (%
\ref{tailbound}), 
\begin{equation*}
\sum_{R(\sigma (\gamma ))>M_{N}}\left\vert \langle \rho _{\alpha }(\gamma
)e_{0}^{\alpha },e_{0}^{\alpha }\rangle \right\vert \leq 4^{\alpha
}80\,\sum_{r>M_{N}}r^{3-\alpha }<L\,M_{N}^{-8}\text{.}
\end{equation*}%
Thus, by Schur's test, 
\begin{equation}
\Vert G-G^{(M_{N})}\Vert \leq L\,M_{N}^{-8}\leq \frac{2^{-N}}{(N+1)^{2}}%
\text{.}  \label{G_G_MN}
\end{equation}%
A sequence of group elements of $\Gamma $ 
\begin{equation*}
\gamma _{0},\gamma _{1},...,\gamma _{j}\text{,}
\end{equation*}%
will be called an\emph{\ orbit path}, if $\gamma _{0}=e$ and 
\begin{equation*}
G^{(M_{N})}\left( \gamma _{k-1},\gamma _{k}\right) \neq 0\text{, \ \ \ }%
1\leq k\leq j\text{.}
\end{equation*}%
By definition of $G^{(M_{N})}$, given an orbit path, 
\begin{equation*}
R\left( \sigma (\gamma _{k-1}^{-1}\gamma _{k})\right) \leq M_{N}\text{, \ \
\ }1\leq k\leq j\text{.}
\end{equation*}%
If we are given two matrices $M_{1},M_{2}\in \limfunc{SL}(2,\mathbb{Z})$,
every entry of the product matrix $M_{1}M_{2}$ is a sum of two products of
the entries of each matrix. Taking the maximum gives 
\begin{equation*}
R(M_{1}M_{2})\leq 2R(M_{1})R(M_{2})\text{.}
\end{equation*}%
Thus,%
\begin{equation}
R\left( \sigma (\gamma _{k})\right) \leq 2R\left( \sigma (\gamma
_{k-1})\right) R\left( \sigma (\gamma _{k-1}^{-1}\gamma _{k})\right) \leq
2M_{N}R\left( \sigma (\gamma _{k-1})\right) \text{.}  \label{Rbound_k}
\end{equation}%
Since $R\left( \sigma (\gamma _{0})\right) =R\left( I\right) $, we can
iterate (\ref{Rbound_k}) to yield 
\begin{equation*}
R\left( \sigma (\gamma _{k})\right) \leq \left( 2M_{N}\right) ^{k}\text{.}
\end{equation*}%
Therefore, writing%
\begin{equation*}
F_{R}=\{\gamma \in \Gamma :R\left( \sigma (\gamma )\right) \leq R\}
\end{equation*}%
we can assure that every vertex of the path belongs to $F_{R_{N}}$, since%
\begin{equation*}
\gamma _{k}\in F_{\left( 2M_{N}\right) ^{k}}\subseteq F_{\left(
2M_{N}\right) ^{N}}=F_{R_{N}},\ \ \ 1\leq k\leq N\text{.}
\end{equation*}%
We conclude that every \emph{orbit path} of length $k\leq N$ beginning at $e$
and using $G^{(M_{N})}$ remains in $\mathcal{F}_{(2M_{N})^{k}}$. As a
result, compression by $P_{R_{N}}$ does not remove any intermediate index
contributing to $\left( G^{(M_{N})}\right) ^{k}\delta _{e}$. Consequently,
extending vectors by zero outside $F_{R_{N}}$,%
\begin{equation*}
\left( B_{N}\right) ^{k}\delta _{e}=\left(
P_{R_{N}}G^{(M_{N})}P_{R_{N}}\right) ^{k}\delta _{e}=\left(
G^{(M_{N})}\right) ^{k}\delta _{e}\text{,\ \ \ }1\leq k\leq N\text{.}
\end{equation*}%
Since $p_{N}$ is a polynomial of degree $N$, it follows that 
\begin{equation*}
q_{N}=p_{N}(B_{N})\delta _{e}=p_{N}(G^{(M_{N})})\delta _{e}\text{.}
\end{equation*}%
Now, $\Xi _{N}=Aq_{N}=Ap_{N}(G^{(M_{N})})\delta _{e}$ 
\begin{equation*}
\Xi _{N}-P_{N}=A\left( p_{N}(G^{(M_{N})})-p_{N}(G)\right) \delta _{e}\text{,}
\end{equation*}%
As a result,%
\begin{equation*}
\left\Vert \Xi _{N}-P_{N}\right\Vert _{\mathcal{A}_{\alpha }}\leq \left\Vert
A\right\Vert _{\mathcal{A}_{\alpha }}\left\Vert
p_{N}(G^{(M_{N})})-p_{N}(G)\right\Vert _{\mathcal{A}_{\alpha }}\text{.}
\end{equation*}%
By the definition (\ref{polynomialEffective}) of $p_{N}$, 
\begin{equation}
\left\Vert p_{N}(B_{N})-Ap_{N}(G)\right\Vert _{\mathcal{A}_{\alpha }}\leq
L^{-1/2}\sum_{j=0}^{N}\binom{2j}{j}4^{-j}\left\Vert \left( I-\frac{%
G^{(M_{N})}}{L}\right) ^{j}-\left( I-\frac{G}{L}\right) ^{j}\right\Vert 
\text{.}  \label{differenceNorm}
\end{equation}%
To estimate the summands in the right hand side, we use, for $j\geq 1$, the
power perturbation inequality%
\begin{equation*}
\left\Vert X^{j}-Y^{j}\right\Vert \leq j\max \left\{ \left\Vert X\right\Vert
,\left\Vert Y\right\Vert \right\} ^{j-1}\left\Vert X-Y\right\Vert
\end{equation*}%
to yield 
\begin{equation}
\left\Vert \left( I-\frac{G^{(M_{N})}}{L}\right) ^{j}-\left( I-\frac{G}{L}%
\right) ^{j}\right\Vert \leq j\left\Vert I-\frac{G^{(M_{N})}}{L}\right\Vert
^{j-1}\left\Vert \frac{G-G^{(M_{N})}}{L}\right\Vert \text{.}  \label{power}
\end{equation}%
Since $0\leq G<L\,I$, the spectrum of $I-\frac{G}{L}$ lies on $\left[ 0,1%
\right] $, hence 
\begin{equation*}
\left\Vert I-\frac{G}{L}\right\Vert \leq 1\text{.}
\end{equation*}%
Now write%
\begin{equation*}
I-\frac{G^{(M_{N})}}{L}=I-\frac{G}{L}+\frac{G-G^{(M_{N})}}{L}\text{.}
\end{equation*}%
Thus, writing $d_{N}=\frac{2^{-N}}{(N+1)^{2}}$ for the bound (\ref{G_G_MN}%
),\ 
\begin{equation*}
\left\Vert I-\frac{G^{(M_{N})}}{L}\right\Vert \leq \left\Vert I-\frac{G}{L}%
\right\Vert +\left\Vert \frac{G-G^{(M_{N})}}{L}\right\Vert \leq 1+\frac{d_{N}%
}{L}\text{.}
\end{equation*}%
Combining this with (\ref{differenceNorm}),\ (\ref{power}) and (\ref{G_G_MN}%
), gives 
\begin{eqnarray*}
\left\Vert p_{N}(B_{N})-p_{N}(G)\right\Vert _{op} &\leq
&L^{-1/2}\sum_{j=0}^{N}\binom{2j}{j}4^{-j}j\frac{d_{N}}{L}\left( 1+\frac{%
d_{N}}{L}\right) ^{j-1} \\
&\leq &2d_{N}L^{-3/2}\sum_{j=0}^{N}j\leq L^{-3/2}2^{-N}\text{,}
\end{eqnarray*}%
where in the second line we used $\binom{2j}{j}4^{-j}\leq 1$ and 
\begin{equation*}
\left( 1+\frac{d_{N}}{L}\right) ^{j-1}\leq \exp \left( (j-1)\frac{d_{N}}{L}%
\right) \leq \exp \left( N\frac{d_{N}}{L}\right) \leq e^{1/8}\leq 2\text{.}
\end{equation*}%
Now, since $\Vert A\Vert ^{2}=\Vert G\Vert <L$, then $\Vert A\Vert \leq 
\sqrt{L}$, , and:%
\begin{equation}
\left\Vert \Xi _{N}-P_{N}\right\Vert _{\mathcal{A}_{\alpha }}\leq \sqrt{L}%
\left\Vert p_{N}(B_{N})-p_{N}(G)\right\Vert _{op}\leq L^{-1}2^{-N}\leq 2^{-N}%
\text{.}  \label{Estimate}
\end{equation}%
Combining (\ref{Estimate}) with (\ref{etacertificate}) gives%
\begin{eqnarray*}
\left\Vert \Phi -\Xi _{N}\right\Vert _{\mathcal{A}_{\alpha }} &\leq
&\left\Vert \Phi -P_{N}\right\Vert _{\mathcal{A}_{\alpha }}+\left\Vert \Xi
_{N}-P_{N}\right\Vert _{\mathcal{A}_{\alpha }} \\
&\leq &\sqrt{1-\Vert P_{N}\Vert _{\mathcal{A}_{\alpha }}^{2}}+2^{-N} \\
&\leq &\sqrt{1-\left( \max \left\{ 0,\left\Vert \Xi _{N}\right\Vert _{%
\mathcal{A}_{\alpha }}-2^{-N}\right\} \right) ^{2}}+2^{-N}\text{,}
\end{eqnarray*}%
since%
\begin{equation}
\Vert P_{N}\Vert _{\mathcal{A}_{\alpha }}\geq \left\Vert \Xi _{N}\right\Vert
_{\mathcal{A}_{\alpha }}-\left\Vert \Xi _{N}-P_{N}\right\Vert _{\mathcal{A}%
_{\alpha }}\geq \left\Vert \Xi _{N}\right\Vert _{\mathcal{A}_{\alpha
}}-2^{-N}\text{.}  \label{P_N_lower}
\end{equation}%
To obtain (\ref{Xieta}) observe that,\ since $q_{N}$ is supported in $%
\mathcal{F}_{R_{N}}$,\ the norm $\left\Vert \Xi _{N}\right\Vert _{\mathcal{A}%
_{\alpha }}$ can be computed exactly from%
\begin{equation}
\left\Vert \Xi _{N}\right\Vert _{\mathcal{A}_{\alpha }}^{2}=q_{N}^{\ast
}G_{R_{N}}q_{N}=\nu _{N}^{2}\text{,}  \label{ExactNorm}
\end{equation}%
where $G_{R_{N}}$ is the full Gram matrix on $\mathcal{F}_{R_{N}}$.

By the reverse triangle inequality and (\ref{Estimate}), 
\begin{equation}
\left\vert \nu _{N}-\Vert P_{N}\Vert _{\mathcal{A}_{\alpha }}\right\vert
\leq \left\Vert \Xi _{N}-P_{N}\right\Vert _{\mathcal{A}_{\alpha }}\leq 2^{-N}
\label{ModuleEstimate}
\end{equation}%
hence, 
\begin{equation*}
0\leq \left( \max \{0,\nu _{N}-2^{-N}\}\right) \leq \Vert P_{N}\Vert _{%
\mathcal{A}_{\alpha }}\leq 1\text{.}
\end{equation*}%
Finally,%
\begin{equation*}
\Vert P_{N}\Vert _{\mathcal{A}_{\alpha }}^{2}=\left\Vert
G^{1/2}p_{N}(G)\delta _{e}\right\Vert _{\ell
^{2}}^{2}=\int_{[0,L]}\left\vert \sqrt{\lambda }\,p_{N}(\lambda )\right\vert
^{2}d\nu _{\delta _{e}}\rightarrow 1\text{.}
\end{equation*}%
Dominated convergence gives $\Vert P_{N}\Vert _{\mathcal{A}_{\alpha
}}\rightarrow 1$. Now, (\ref{ModuleEstimate}) implies $\nu _{N}\rightarrow 1$%
. Consequently,%
\begin{equation*}
\max \{0,\nu _{N}-2^{-N}\}\rightarrow 1\text{.}
\end{equation*}%
and%
\begin{equation*}
\sqrt{1-\left( \max \{0,\nu _{N}-2^{-N}\}\right) ^{2}}+2^{-N}\rightarrow 0%
\text{.}
\end{equation*}%
Finally, since $\nu _{N}^{2}=q_{N}^{\ast }G_{R_{N}}q_{N}$\ is computed from
finite matrices, the bound (\ref{effective}) is a finite computable error
certificate. The inequality (\ref{pointwise}) follows trivially from the
reproducing kernel identity. \hfill $\square $

\subsection{The effective procedure}

Theorem \ref{finitecertificate} provides a so-called effective procedure to
compute $\Phi $, showing therefore that $\Phi $\emph{\ is computable} in the
standard sense of computable analysis \cite{Computability}. Indeed, we may
equip $\mathcal{A}_{12}$ with the computable presentation generated by the
explicit orbit vectors $\rho _{12}(\gamma )e_{0}^{\alpha }$, $\gamma \in
\Gamma $. Since $\Xi _{N}$ is a finite algebraic combination of these
vectors and $\left\Vert \Phi -\Xi _{N}\right\Vert \leq E_{N}\rightarrow 0$
effectively, where 
\begin{equation*}
E_{N}=2^{-N}+\sqrt{1-\left( \max \{0,\nu _{N}-2^{-N}\}\right) ^{2}}\text{.}
\end{equation*}%
the magic function $\Phi $ is a computable element of $\mathcal{A}_{12}$. We
explicitly outline the 7 steps of the explicit approximation algorythm,
which provides an effective procedure.

\textbf{Effective procedure. }Given $\epsilon >0$, proceed as follows:

\begin{enumerate}
\item Chose $N=1,2,...$

\item Compute $d_{N},M_{N},R_{N}$.

\item Enumerate the finite set%
\begin{equation*}
F_{R_{N}}=\{\gamma \in \Gamma :R\left( \sigma (\gamma )\right) \leq R_{N}\}
\end{equation*}

\item Construct the finite matrices $B_{N}$ and $G_{R_{N}}$.

\item Compute 
\begin{equation*}
q_{N}=p_{N}(B_{N})\delta _{e}
\end{equation*}%
and the finite function $\Xi _{N}=Aq_{N}$.

\item Compute%
\begin{equation*}
\nu _{N}=(q_{N}^{\ast }G_{R_{N}}q_{N})^{1/2}
\end{equation*}%
and 
\begin{equation*}
E_{N}=2^{-N}+\sqrt{1-\left( \max \{0,\nu _{N}-2^{-N}\}\right) ^{2}}\text{.}
\end{equation*}

\item Output $\Xi _{N}$ when $E_{N}<\varepsilon $. Otherwise, increment $N$.
\end{enumerate}

The operations involved in 1.-7. are finite: enumeration of integer
matrices, construction of finite Gram matrices, matrix--vector arithmetic,
evaluation of a polynomial at a finite matrix, and certified algebraic
square-root comparison. We have chosen the spectral bound $L=4^{12}100$ so
that $L^{1/2}=40960$ is an integer and $L^{-1/2}$ a rational number.
Consequently, all entries of the finite Gram matrices and all coefficients
of the approximants $\Xi _{N}$ belong to $\mathbb{Q}(\zeta _{12})$ and can
be obtained by exact algebraic arithmetic. The associated norm and error
certificates involve only finitely many arithmetic operations, comparisons,
positive square roots of explicitly given algebraic numbers and are
therefore effectively computable.

\textbf{Acknowledgement:} This research was funded in part by the Austrian
Science Fund (FWF) through the project 10.55776/PAT8205923 (L.D.A.). AI
assistance proved to be of valuable help in the development of the effective
procedure presented in section 4. For open access purposes, the author has
applied a CC BY public copyright license to any author-accepted manuscript
version arising from this submission.

\end{document}